\documentclass[oneside,article]{memoir}
\usepackage{amsmath}         
\usepackage{amsthm,thmtools,thm-restate} 
\usepackage{enumitem}        

\usepackage{tikz}
  \usetikzlibrary{matrix,arrows,positioning}
  
\usepackage{tikz-cd}

\usepackage{float} 
\usepackage{caption, subcaption} 

\usepackage[nobysame,alphabetic,initials]{amsrefs} 
\usepackage[hidelinks]{hyperref}                   
\usepackage[capitalize,nosort]{cleveref}           

\usepackage{amssymb}
\usepackage[mathscr]{euscript}
\usepackage{relsize}
\usepackage{stmaryrd}
\usepackage{stackengine}

\usepackage{indentfirst} 
\usepackage{multicol}

\usepackage[inline,nomargin]{fixme} 
  \fxsetup{targetlayout=color}

\usepackage{float} 

\usepackage{lmodern}
\usepackage{libertine}

\isopage[12]

\setlrmargins{*}{*}{1}
\checkandfixthelayout

\counterwithout{section}{chapter}
\usepackage{appendix}

  \newcommand{\adj}{\dashv}

  \newcommand{\nerve}{\operatorname{N}}

  \newcommand{\gprod}{\otimes}

  \newcommand{\str}{\mathfrak{C}}

  \newcommand{\id}[1][]{\operatorname{id}_{#1}}
  
  \newcommand{\simp}[1]{\mathord\Delta^{#1}}

  \newcommand{\horn}[2]{\Lambda^{#1}_{#2}}

  \newcommand{\ncat}[1]{\mathsf{#1}}

  \newcommand{\from}{\colon}

  \declaretheorem[style=definition]{definition}
  
  \declaretheorem[style=definition,numberlike=definition]{remark}

  \declaretheorem[style=plain,numberlike=definition]{corollary}
  \declaretheorem[style=plain,numberlike=definition]{lemma}
  
  \declaretheorem[style=plain,numberlike=definition]{theorem}

  \declaretheorem[style=plain,numbered=no,name=Theorem]{theorem*}

  \Crefname{corollary}{Corollary}{Corollaries}
  \Crefname{definition}{Definition}{Definitions}
  \Crefname{lemma}{Lemma}{Lemmas}
  \Crefname{proposition}{Proposition}{Propositions}
  \Crefname{remark}{Remark}{Remarks}
  \Crefname{theorem}{Theorem}{Theorems}
  \Crefname{notation}{Notation}{Notations}

  \newlist{axioms}{enumerate}{1}
  \Crefname{axiomsi}{}{}

  \newenvironment{tikzeq*}
  {
    \begingroup
    \begin{equation*}
    \begin{tikzpicture}[baseline=(current bounding box.center)]
  }
  {
    \end{tikzpicture}
    \end{equation*}
    \endgroup
    \ignorespacesafterend
  }

  \tikzset
  {
    diagram/.style=
    {
      matrix of math nodes,
      column sep={4.3em,between origins},
      row sep={4em,between origins},
      text height=1.5ex,
      text depth=.25ex
    },
    over/.style={preaction={draw=white,-,line width=6pt}},
    every to/.style={font=\footnotesize},
    inj/.style={right hook->},
    surj/.style={-{Latex[open]}},
    cof/.style={>->},
    fib/.style={->>},
  }

  \DeclareFontFamily{U}{mathx}{\hyphenchar\font45}

  \DeclareFontShape{U}{mathx}{m}{n}{
    <5> <6> <7> <8> <9> <10>
    <10.95> <12> <14.4> <17.28> <20.74> <24.88>
    mathx10}{}

  \DeclareSymbolFont{mathx}{U}{mathx}{m}{n}

  \DeclareFontFamily{U}{mathb}{\hyphenchar\font45}

  \DeclareFontShape{U}{mathb}{m}{n}{
    <5> <6> <7> <8> <9> <10>
    <10.95> <12> <14.4> <17.28> <20.74> <24.88>
    mathb10}{}

  \DeclareSymbolFont{mathb}{U}{mathb}{m}{n}

  \DeclareMathAccent{\widebar}{0}{mathx}{"73}

  \DeclareMathSymbol{\Rsh}{\mathrel}{mathb}{"E9}

  \DeclareFontFamily{U}{MnSymbolA}{}

  \DeclareFontShape{U}{MnSymbolA}{m}{n}{
    <-6> MnSymbolA5
    <6-7> MnSymbolA6
    <7-8> MnSymbolA7
    <8-9> MnSymbolA8
    <9-10> MnSymbolA9
    <10-12> MnSymbolA10
    <12-> MnSymbolA12}{}

  \DeclareSymbolFont{MnSyA}{U}{MnSymbolA}{m}{n}

  \DeclareMathSymbol{\twoheaddownarrow}{\mathrel}{MnSyA}{27}

  \newcommand{\MSC}[1]{%
    \let\thempfn\relax
    \footnotetext[0]{2020 Mathematics Subject Classification: #1.}
  }

\tikzstyle{vertex}=[circle, fill, minimum size=4pt, inner sep=0pt]

\newcommand{\cSet}{\ncat{cSet}}
\newcommand{\Graph}{\ncat{Graph}}

\newcommand{\gtimes}{\otimes} 
\newcommand{\gexp}[2]{\ensuremath{#1}^{\gtimes #2}} 
\newcommand{\ghom}[2]{\operatorname{hom}^{\gtimes}(#1, #2)} 

\newcommand{\face}[2]{\partial^{#1}_{#2}} 

\newcommand{\boxcat}{\mathord{\square}} 

\newcommand{\cube}[1]{\mathord{\square^{#1}}} 

\newcommand{\gnerve}[1][]{\mathrm{N}_{#1}} 
\newcommand{\lhom}{\operatorname{hom}_{L}} 
\newcommand{\rhom}{\operatorname{hom}_{R}} 

\newcommand{\join}{\ast} 

\newcommand{\cohnerve}{{\nerve}_{\boxcat}} 

\newcommand{\restr}[2]{{#1}|_{#2}} 

\thanksmarkseries{arabic}

\author{
  Daniel Carranza
  \and Krzysztof Kapulkin 
  \and Jinho Kim
}
\title{Nonexistence of colimits in naive discrete homotopy theory}
\date{\today}

\begin{document}

  \maketitle

\begin{abstract} 
  We show that the quasicategory defined as the localization of the category of (simple) graphs at the class of A-homotopy equivalences does not admit colimits. 
  In particular, we settle in the negative the question of whether the A-homotopy equivalences in the category of graphs are part of a model structure.

    \MSC{05C25 (primary), 55U35, 18N60, 18N50 (secondary)}
\end{abstract}

During a recent workshop ``Discrete and Combinatorial Homotopy Theory'' at the American Institute of Mathematics, it was asked whether A-homotopy equivalences are part of a model structure on the category of simple graphs.
In this short note, we settle this question in the negative by considering the quasicategory associated to the marked category (i.e., a category along with a distinguished wide subcategory) of simple graphs and A-homotopy equivalences does not have pushouts.
The result then follows, since model categories are known to present quasicategories that are both complete and cocomplete.

A-homotopy theory, introduced by Barcelo and collaborators \cites{barcelo-kramer-laubenbacher-weaver,babson-barcelo-longueville-laubenbacher,barcelo-laubenbacher}  based on the earlier work of Atkin \cites{atkin:i,atkin:ii}, uses methods of homotopy theory to study combinatorial properties of graphs.
In studying A-homotopy theory, one considers two classes of maps: A-homotopy equivalences, i.e., maps with inverses up to A-homotopy, and weak A-homotopy equivalences, i.e., maps inducing isomorphisms on all A-homotopy groups
These two choices lead to two distinct homotopy theories, and we refer to the former of those as the \emph{naive} discrete homotopy theory and the latter as the discrete homotopy theory.

A concise, 3-page introduction to A-homotopy theory can be found in \cite[\S12]{carranza-kapulkin-lindsey:calculus-of-fractions} and we will follow the notation established there.
In particular, we write $\Graph$ for the (locally Kan) cubical category of graphs and $\cohnerve \Graph$ for its cubical homotopy coherent nerve.
As observed there, the quasicategory $\cohnerve \Graph$ is the localization of the category of simple graphs at the class of A-homotopy equivalences.
The box product of graphs is denoted $\gtimes$ and its right adjoint by $\ghom{-}{=}$.
The nerve of a graph $G$ is a cubical set $\gnerve G$.
The geometric product of cubical sets is also denoted by $\gprod$.
The right adjoint to the functor $- \gprod X$ (or $X \gprod -$) is denoted by $\lhom(X, -)$ (respectively, $\rhom(X, -)$).

Our main theorem (\cref{cohnerve-graph-no-pushout}) asserts that a certain span does not admit a colimit, revealing an inherent rigidity of A-homotopy equivalences.
From that, we deduce that A-homotopy equivalences are not part of any model or cofibration category structure on the category of simple graphs.
This is of course not an issue with discrete homotopy theory, but instead with the class of A-homotopy equivalences, which for independent reasons is not useful for combinatorial applications either.

Lastly, we mention that Goyal and Santhanam have recently established similar results \cite{goyal-santhanam:i,goyal-santhanam:ii} in the context of Dochtermann's $\times$-homotopy theory \cite{dochtermann:hom-complex-and-homotopy}.
Their techniques are more direct and do not involve higher category theory.
Our proof implies not only that there is no model/cofibration category structure on simple graphs with A-homotopy equivalences, but also that there cannot be such a structure on any marked category DK-equivalent to it.
In other words, we identify a problem with the homotopy theory, not a presentation thereof.





Sepcifically, we consider the following diagram
\[ \begin{tikzcd}
    I_0 \sqcup I_0 \ar[r] \ar[d] & I_0 \\
    I_0
\end{tikzcd} \]
which we denote by $D \from \horn{2}{0} \to \cohnerve\Graph$.
We will show that this diagram does not admit a pushout.
Informally, the outline of our argument is as follows.
\begin{enumerate}
    \item We give an explicit description of the objects and morphisms (i.e.~the 1-skeleton) of the quasicategory of cones under $D$.
    A cone under $D$ may be thought of as a graph $G$ and a cycle in $G$.
    A morphism between cones consists of a graph map $f \from G \to H$ and a homotopy between the induced cycles in $H$.
    \item This description then shows there is no initial cone under $D$: given a cycle in a graph $G$ of length $m$, its image under any map $G \to C_{m+1}$ must be null-homotopic (where $C_{m+1}$ denotes the cycle graph with $m+1$ vertices).
    This implies there can be no homotopy from the image of this cycle to the tautological cycle in the graph $C_{m+1}$.
    That is, for every cone $\lambda$ under $D$, there exists a cone $\lambda'$ for which there is no morphism $\lambda \to \lambda'$.
\end{enumerate}
The formal statement and proof of (1) may be found in \cref{cone-D-obj,cone-D-mor}.
The precise formulation of (2) is given by \cref{cohnerve-graph-no-pushout}.

We begin with a short calculation.
\begin{lemma} \label{n-cube-hom-coprod-pt}
    For any graph $G$ and set $S$, we have an isomorphism
    \[ \cSet(\cube{n}, \gnerve \ghom{S \times I_0}{G}) \cong \cSet(S \times \cube{n}, \gnerve G) \]
    natural in $G$.
\end{lemma}
\begin{proof}
    We calculate
    \begin{align*}
        \cSet(\cube{n}, \gnerve \ghom{S \times I_0}{G}) 
        &\cong \cSet(\cube{n}, \lhom(S \times \cube{0}, \gnerve G)) \text{ by \cite[Prop.~3.10]{carranza-kapulkin:cubical-graphs}} \\
        &\cong \cSet(\cube{n} \gprod (S \times \cube{0}), \gnerve G) \\
        &\cong \cSet(S \times (\cube{n} \gprod \cube{0}), \gnerve G) \\
        &\cong \cSet(S \times \cube{n}, \gnerve G). \qedhere
    \end{align*}
\end{proof}

For a graph map $f \from I_\infty \to G$ that is stable in the positive (or negative) direction, we write $f(\infty)$ (or $f(-\infty)$, respectively) for the value of $f$ in the stable range.
\begin{theorem} \label{cone-D-obj}
    There is a bijection between the set of objects of the quasicategory of cones under $D$ and the set of tuples $(G, p_1, p_2, p_3, p_4)$, where
    \begin{itemize}
        \item $G$ is a graph; and
        \item $p_1, p_2, p_3$ and $p_4$ are maps $I_\infty \to G$ which are stable in both directions such that
        \[ \begin{array}{l l l l}
            p_1(\infty) = p_2(\infty) & p_1(-\infty) = p_3(-\infty) &
            p_2(-\infty) = p_4(-\infty) & p_3(\infty) = p_4(\infty).
        \end{array} \]
    \end{itemize}
\end{theorem}
\begin{proof}
    As $\horn{2}{0} \join \simp{0} \cong \simp{1} \times \simp{1}$, a cone under $D$ is a square $\lambda \from \simp{1} \times \simp{1} \to \cohnerve\Graph$ such that the maps $\restr{\lambda}{\id \times 0}, \restr{\lambda}{0 \times \id} \from \simp{1} \to \cohnerve\Graph$ are both $I_0 \sqcup I_0 \to I_0$.
    Identifying $\simp{1} \times \simp{1}$ as the pushout
    \[ \begin{tikzcd}
        \simp{1} \ar[r, "{\face{}{1}}"] \ar[d, "{\face{}{1}}"'] \ar[rd, phantom, "\ulcorner" very near end] & \simp{2} \ar[d] \\
        \simp{2} \ar[r] & \simp{1} \times \simp{1}
    \end{tikzcd} \]
    by the $\str \adj \cohnerve$-adjunction, this corresponds to a commutative square
    \[ \begin{tikzcd}
        \str[1] \ar[r, "\face{}{1}"] \ar[d, "\face{}{1}"'] & \str[2] \ar[d] \\
        \str[2] \ar[r] & \Graph
    \end{tikzcd} \]
    i.e.~two cubical functors $T, U \from \str[2] \to \Graph$ such that $T\face{}{2}, U\face{}{2} \from \str[1] \to \Graph$ are both the map $I_0 \sqcup I_0 \to I_0$.
    We depict $T$ and $U$ as two homotopy-commutative triangles in the diagram,
    \[ \begin{tikzcd}[sep = large]
        I_0 \sqcup I_0 \ar[r] \ar[d] \ar[rd, "{[v_1,v_3]}"{name=M, description}] & I_0 \ar[d, "v_2"] \ar[from=M, Rightarrow, "p_{12}"] \\
        I_0 \ar[r, "v_4"'] \ar[from=M, Rightarrow, "p_{34}"' very near start] & G
    \end{tikzcd} \]
    where $p_{12} \from \cube{1} \to \gnerve \ghom{I_0 \sqcup I_0}{G}$ denotes the image of the non-degenerate 1-cube in $\str[2](0, 2)$ under $T$, and $p_{34} \from \cube{1} \to \gnerve \ghom{I_0 \sqcup I_0}{G}$ denotes its image under $U$.
    By \cref{n-cube-hom-coprod-pt}, we may identify $p_{12}$ and $p_{34}$ with maps $[p_1, p_2], [p_3, p_4] \from \cube{1} \sqcup \cube{1} \to \gnerve G$ respectively, and naturality gives that:
    \[ \begin{array}{c c c c}
        p_1\face{}{1,0} = v_1 & p_1\face{}{1,1} = v_2 & p_2\face{}{1,0} = v_3 & p_2\face{}{1,1} = v_2 \\
        p_3\face{}{1,0} = v_1 & p_3\face{}{1,1} = v_4 & p_4\face{}{1,0} = v_3 & p_4\face{}{1,1} = v_4.
    \end{array} \]
    By definition, the 1-cubes $p_1, p_2, p_3, p_4 \from \cube{1} \to \gnerve G$ are stable maps $I_\infty \to G$ that satisfy the desired endpoint equalities.
\end{proof}
\begin{remark} \label{rmk:cone-D-is-cycle}
    Let $(G, p_1, p_2, p_3, p_4)$ be a cone under $D$ (under the bijection in \cref{cone-D-obj}).
    One may identify each $p_i \from I_\infty \to G$ with a map $I_m \to G$ where $m$ is the smallest positive integer which makes $p_i$ stable.
    With this, the concatenation $p_1 \cdot p_2^{-1} \cdot p_4 \cdot p_3^{-1}$ gives a cycle in $G$,
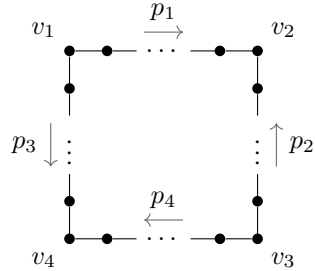
    \begin{center}
        \begin{tikzpicture}[scale=0.5]
            \node[vertex, label={above left:$v_1$}] (TL) {}; 
            \path[draw] (TL) -- +(1,0) node[vertex] (T1) {};
            \path (T1) -- +(1.5,0) node[minimum width=0.7cm,label={[label distance=0.4em]:$p_1$}] (Tdots) {$\dots$};
            \draw (T1) -- (Tdots);
            \path[draw] (Tdots) -- +(1.5,0) node[vertex] (T2) {};
            \path[draw] (T2) -- +(1,0) node[vertex, label={above right:$v_2$}] (TR) {};
            \path (Tdots.west) -- +(0.2, 0.5) coordinate (TarrowL);
            \path (Tdots.east) -- +(-0.2, 0.5) coordinate (TarrowR);
            \path[->, draw=black!60] (TarrowL) -- (TarrowR);

            \path[draw] (TR) -- +(0, -1) node[vertex] (R1) {};
            \path (R1) -- +(0, -1.5) node[minimum height=0.7cm, label={[label distance=0.4em]right:$p_2$}] (Rdots) {$\vdots$};
            \draw (R1) -- (Rdots);
            \path[draw] (Rdots) -- +(0, -1.5) node[vertex] (R2) {};
            \path[draw] (R2) -- +(0, -1) node[vertex, label={below right:$v_3$}] (BR) {};
            \path (Rdots.north) -- +(0.5, -0.2) coordinate (RarrowT);
            \path (Rdots.south) -- +(0.5, 0.2) coordinate (RarrowB);
            \path[->, draw=black!60] (RarrowB) -- (RarrowT);

            \path[draw] (TL) -- +(0, -1) node[vertex] (L1) {};
            \path (L1) -- +(0, -1.5) node[minimum height=0.7cm, label={[label distance=0.4em]left:$p_3$}] (Ldots) {$\vdots$};
            \draw (L1) -- (Ldots);
            \path[draw] (Ldots) -- +(0, -1.5) node[vertex] (L2) {};
            \path[draw] (L2) -- +(0, -1) node[vertex, label={below left:$v_4$}] (BL) {};
            \path (Ldots.north) -- +(-0.5, -0.2) coordinate (LarrowT);
            \path (Ldots.south) -- +(-0.5, 0.2) coordinate (LarrowB);
            \path[->, draw=black!60] (LarrowT) -- (LarrowB);

            \path[draw] (BL) -- +(1,0) node[vertex] (B1) {};
            \path (B1) -- +(1.5,0) node[minimum width=0.7cm, label={[label distance=0.4em]:$p_4$}] (Bdots) {$\dots$};
            \draw (B1) -- (Bdots);
            \path[draw] (Bdots) -- +(1.5,0) node[vertex] (B2) {};
            \path[draw] (B2) -- (BR);
            \path (Bdots.west) -- +(0.2, 0.5) coordinate (BarrowL);
            \path (Bdots.east) -- +(-0.2, 0.5) coordinate (BarrowR);
            \path[->, draw=black!60] (BarrowR) -- (BarrowL);
        \end{tikzpicture}
    \end{center}
    thus justifying the intuition that a cone under $D$ corresponds to a cycle in a graph $G$.
    However, this process does not give a well-defined function from cones under $D$ to tuples $(G \in \Graph, f \from C_m \to G)$.
    For instance, if one modifies the cycle $p_1 \cdot p_2^{-1} \cdot p_4 \cdot p_3^{-1}$ by repeating any endpoint of some $p_i$ then one obtains a different cycle $C_n \to G$ which corresponds to the same cone under $D$.
    However, each $p_i$ induces a well-defined path up to homotopy, thus the cycle $p_1 \cdot p_2^{-1} \cdot p_4 \cdot p_3^{-1}$ represents a well-defined element of the discrete fundamental group of $G$.
\end{remark}
\begin{theorem} \label{cone-D-mor}
    Under the bijection of \cref{cone-D-obj}, let $(G, p_1, p_2, p_3, p_4)$ and $(H, q_1, q_2, q_3, q_4)$ be two cones under $D$.
    There is a bijection between the set of 1-simplices in in the quasicategory of cones under $D$ from $(G, p_1, p_2, p_3, p_4)$ to $(H, q_1, q_2, q_3, q_4)$ and the set of tuples $(f, \alpha_1, \alpha_2, \alpha_3, \alpha_4)$ where
    \begin{itemize}
        \item $f$ is a graph map $G \to H$; and
        \item $\alpha_i$ is a homotopy $\gexp{I_\infty}{2} \to H$ from $f \circ p_i$ to $q_i$ for $i = 1, 2, 3, 4$.
    \end{itemize}
\end{theorem}
\begin{proof}
    A map $\lambda \to \lambda'$ between cones under $D$ is a map $\alpha \from \horn{2}{0} \join \simp{1} \to \cohnerve\Graph$ such that $\restr{\alpha}{\horn{2}{0} \join \{ 0 \}} = \lambda$ and $\restr{\alpha}{\horn{2}{0} \join \{ 1 \}} = \lambda'$.
    Writing $\horn{2}{0} \join \simp{1}$ as a pushout:
    \[ \begin{tikzcd}
        \simp{2} \ar[r, "\face{}{1}"] \ar[d, "\face{}{1}"'] \ar[rd, phantom, "\ulcorner" very near end] & \simp{3} \ar[d] \\
        \simp{3} \ar[r] & \horn{2}{0} \join \simp{1}
    \end{tikzcd} \]
    as before, this corresponds to the data of two cubical diagrams $T, U \from \str[3] \to \Graph$, which we depict below
    
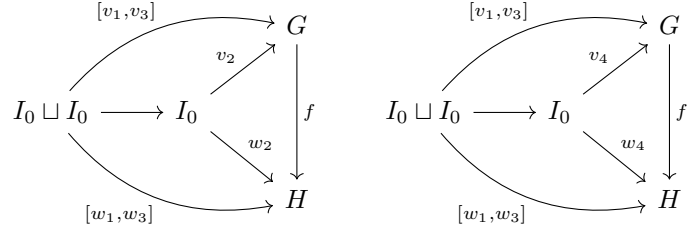
\[ \begin{tikzcd}
        {} & {} & G \ar[dd, "f"] \\
        I_0 \sqcup I_0 \ar[r] \ar[urr, bend left, "{[v_1, v_3]}"] \ar[drr, bend right, "{[w_1, w_3]}"'] & I_0 \ar[ur, "v_2"] \ar[dr, "w_2"] & {} \\
        {} & {} & H
    \end{tikzcd} \qquad \begin{tikzcd}
        {} & {} & G \ar[dd, "f"] \\
        I_0 \sqcup I_0 \ar[r] \ar[urr, bend left, "{[v_1, v_3]}"] \ar[drr, bend right, "{[w_1, w_3]}"'] & I_0 \ar[ur, "v_4"] \ar[dr, "w_4"] & {} \\
        {} & {} & H
    \end{tikzcd} \]
    (note our depiction omits the data of the 1-cubes and 2-cube in the mapping spaces of $\str[3]$).
    Let $\alpha_{12}, \alpha_{34} \from \cube{2} \to \gnerve \ghom{I_0 \sqcup I_0}{H}$ denote the images of the non-degnerate 2-cube in $\str[3](0, 3)$ under $T$ and $U$, respectively.
    By \cref{n-cube-hom-coprod-pt}, we identify $\alpha_{12}$ and $\alpha_{34}$ with maps $[\alpha_1, \alpha_2], [\alpha_3, \alpha_4] \from \cube{2} \sqcup \cube{2} \to \gnerve H$ and apply naturality to identify the faces of each square.
    We depict these 2-cubes below:
    \[ \begin{tikzcd}
        w_1 \ar[r, "q_1"] \ar[d] \ar[rd, phantom, "\alpha_1" description] & w_2 \ar[d] \\
        fv_1 \ar[r, "f \circ p_1"'] & fv_2
    \end{tikzcd} \quad \begin{tikzcd}
        w_3 \ar[r, "q_2"] \ar[d] \ar[rd, phantom, "\alpha_2" description] & w_2 \ar[d] \\
        fv_3 \ar[r, "f \circ p_2"'] & fv_2
    \end{tikzcd} \quad \begin{tikzcd}
        w_1 \ar[r, "q_3"] \ar[d] \ar[rd, phantom, "\alpha_3" description] & w_4 \ar[d] \\
        fv_1 \ar[r, "f \circ p_3"'] & fv_4
    \end{tikzcd} \quad \begin{tikzcd}
        w_3 \ar[r, "q_4"] \ar[d] \ar[rd, phantom, "\alpha_4" description] & w_4 \ar[d] \\
        fv_3 \ar[r, "f \circ p_4"'] & fv_4
    \end{tikzcd} \]
\end{proof}

\begin{theorem} \label{cohnerve-graph-no-pushout}
    For any cone $\lambda$ under $D$, there exists a cone $\lambda'$ under $D$ such that there is no cone map $\lambda \to \lambda'$.
    In particular, the diagram $D$ does not have a pushout in $\cohnerve\Graph$.
\end{theorem}
\begin{proof}
    Fix a cone $\lambda$ under $D$.
    By \cref{cone-D-obj}, we identify $\lambda$ as a tuple $(G, p_1, p_2, p_3, p_4)$.
    As the concatenation $p_1 \cdot p_2^{-1} \cdot p_4 \cdot p_3^{-1}$ represents a well-defined element of $A_1(G, v_1)$ (see \cref{rmk:cone-D-is-cycle}), let $\varphi \from C_m \to G$ be a cycle of minimal length which represents this concatenation.

    Let $\psi \from I_\infty \to C_{m+5}$ denote the identity cycle on $C_{m+5}$ based at 0 and $\mathsf{const}_0 \from I_\infty \to C_{m+5}$ denote the constant map at 0.
    By \cref{cone-D-obj}, the tuple $(C_{m+5}, \psi, \mathsf{const}_0, \mathsf{const}_0, \mathsf{const}_0)$ is a cone over $D$.
    We assume there exists a morphism of cones from $(G, p_1, p_2, p_3, p_4)$ to $(C_{m+5}, \psi, \mathsf{const}_0, \mathsf{const}_0, \mathsf{const}_0)$ and derive a contradiction.

    By \cref{cone-D-mor}, this morphism corresponds to a tuple $(f, \alpha_1, \alpha_2, \alpha_3, \alpha_4)$.
    The horizontal concatenation $\alpha_1 \cdot \alpha_2^{-1} \cdot \alpha_4 \cdot \alpha_3^{-1}$ induces a square
    \[ \begin{tikzcd}
        0 \ar[r, "\psi"] \ar[d, "\gamma"'] & 0 \ar[d, "\gamma"] \\
        fv_1 \ar[r, "f \circ \varphi"] & fv_1
    \end{tikzcd} \]
    where $\gamma$ denotes the restriction $\restr{\alpha_1}{(-\infty, -)} \from I_\infty \to C_{m+5}$ of the square $\alpha_1$ to its left leg.
    The cycle $f \circ \varphi$ is (based) homotopic to the constant path as it has length at most $m$.
    Thus, this square yields a chain of based homotopies
    \[ \gamma \sim \gamma \cdot (f \circ \varphi) \sim \psi \cdot \gamma . \]
    This implies $\psi$ is based homotopic to the constant cycle, which is a contradiction as $C_{m+5}$ is a cycle of length greater than 5.
\end{proof}

\begin{corollary}
  The class of A-homotopy equivalences is not part of a model structure or a cofibration category structure on the category of simple graphs.
\end{corollary}

\begin{proof}
  By the preceding theorem, the localization of simple graphs at the class of A-homotopy equivalences does not admit pushouts.
  Since the localization of a model category or a cofibration category at its weak equivalences must be cocomplete \cite{kapulkin-szumilo:frames,szumilo:frames}, the result follows.
\end{proof}

\textbf{Acknowledgements.} We thank the anonymous referee for insightful comments that helped restructure the paper and make arguments cleaner.

This material is based upon work supported by the National Science Foundation under Grant No.~DMS-1928930 while the first two authors participated in a program hosted by the Mathematical Sciences Research Institute in Berkeley, California, during the 2022--23 academic year.






  

\bibliographystyle{amsalphaurlmod}
\bibliography{discrete-homotopy-theory.bib}

\end{document}